\documentclass[a4paper,12pt,oneside]{article}
\usepackage{amsmath, amsthm, amssymb, hyperref}

\usepackage{geometry}
\usepackage{enumitem}
\setlist{noitemsep}

\usepackage{listings}
\usepackage{xcolor}

\definecolor{codeblue}{rgb}{0.3,0.3,0.6}
\definecolor{codegreen}{rgb}{0,0.6,0}
\definecolor{codegray}{rgb}{0.5,0.5,0.5}
\definecolor{codepurple}{rgb}{0.58,0,0.82}
\definecolor{backcolour}{rgb}{0.95,0.95,0.92}

\lstdefinestyle{mystyle}{
    backgroundcolor=\color{backcolour},   
    commentstyle=\color{codegreen},
    keywordstyle=\color{codeblue},
    numberstyle=\tiny\color{codegray},
    stringstyle=\color{codepurple},
    basicstyle=\ttfamily\footnotesize,
    breakatwhitespace=false,         
    breaklines=true,                 
    captionpos=b,                    
    keepspaces=true,                 
    numbers=left,                    
    numbersep=5pt,                  
    showspaces=false,                
    showstringspaces=false,
    showtabs=false,                  
    tabsize=2
}

\usepackage{amsthm,amsmath,amsfonts,amssymb}
\usepackage{xcolor}
   \usepackage{geometry}
\usepackage{authblk}

\newgeometry{vmargin={28mm}, hmargin={20mm,17mm}}   

\makeatletter

\hypersetup{
    colorlinks,
    linkcolor={red},
    citecolor={green},
    urlcolor={blue}
}

\newtheorem{theorem}{Theorem}
\newtheorem{lem}[theorem]{Lemma}

\newtheorem{prop}[theorem]{Proposition}
\newtheorem{rem}[theorem]{Remark}

\newtheorem{maintheorem}{Theorem}

\usepackage{verbatim}
\usepackage{color}
\usepackage{authblk}
\hypersetup{
    colorlinks,
    linkcolor={orange},
    citecolor={magenta},
    urlcolor={blue}
}

\title{Finite groups with exactly $2p+1$ subgroups of order $p$}

\author[1]{Ihechukwu Chinyere\thanks{Email: \texttt{i.chinyere@up.ac.za}}}
\author[2]{Bernard O. Bainson\thanks{Corresponding author: \texttt{bboduoku@knust.edu.gh}}}
\affil[1]{Department of Mathematics and Applied Mathematics, University of Pretoria, Hatfield 0028, Pretoria, South Africa}
\affil[2]{Department of Mathematics, Kwame Nkrumah University of Science and
Technology, Ghana}
\begin{document}
\maketitle
\begin{abstract}
Let $\text{m}_p(G)$ be the number of elements of order $p$ in a finite group $G$, $p$ prime. We prove that there is no finite simple group $G$ with $\text{m}_p(G)=2p^2-p-1$, and that a finite group with this property exists if and only if $2p+1$ is prime or $3^f$ with $f$ an odd prime. This concerns an open problem and extends the classification of finite groups with few elements of prime order.
\end{abstract}

\maketitle

\noindent\textbf{2010 Mathematics Subject Classification.} 20D05, 20D60.\\
\noindent\textbf{Keywords:} Finite groups, finite simple groups, elements of prime order.

\section{Introduction}\label{sec:intro}

Let $G$ be a finite group and $p$ a prime. We write $\text{m}_p(G)$ for the number of elements of order $p$ in $G$ and $\text{s}_p(G)$ for the number of subgroups of order $p$. Since distinct subgroups of order $p$ intersect trivially,
\[
\text{m}_p(G)=(p-1)\,\text{s}_p(G).
\]
By Frobenius' theorem \cite{Frobenius95}, if $p$ divides $|G|$ then $\text{s}_p(G)\equiv 1 \pmod p$. Hence, if $p\mid|G|$, the possible values of $\text{m}_p(G)$ are $(p-1)(kp+1)$, $k=0,1,2,\dots$, namely,
\[
p-1,\quad p^2-1,\quad 2p^2-p-1,\quad 3p^2-2p-1,\ \dots
\]
(If $p\nmid|G|$, then $\text{m}_p(G)=0$.) The first value, $k=0$, occurs exactly when $G$ has a unique subgroup of order $p$, which is then normal; among simple groups this occurs only for $G\cong\mathbb Z_p$, since a simple nonabelian group has no normal subgroup of order $p$. The second value, $k=1$, is achieved by $\mathrm{PSL}_2(p)$ ($p\ge5$) and only by it among simple groups \cite{Asboei26}. For $p=2$, larger values of $k$ already occur among simple groups: $\mathrm{SL}_2(2^n)$, simple for $n\ge 2$, has exactly $2^{2n}-1=(p-1)(kp+1)$ involutions with $k=2^{2n-1}-1$. Asboei \cite{Asboei26} asked for a classification of finite groups with $\text{m}_p(G)=(p-1)(kp+1)$ for $k\ge 2$. We treat $k=2$. Note that $2p^2-p-1=(p-1)(2p+1)$, so condition $\text{m}_p(G)=2p^2-p-1$ is equivalent to $\text{s}_p(G)=2p+1$.

\begin{maintheorem}\label{main}
Let $p$ be a prime. The following are equivalent.
\begin{enumerate}[label=(\alph*)]
\item There is a finite group $G$ with $\text{m}_p(G)=2p^2-p-1$.
\item $2p+1$ is a prime power.
\item $2p+1$ is prime, or $2p+1=3^f$ with $f$ an odd prime.
\end{enumerate}

\end{maintheorem}
Theorem~\ref{main} is complemented by the following result, which shows that unlike the cases $k=0, 1$, the groups in (a) are never simple.

\begin{maintheorem}\label{simple}
Let $p$ be a prime. There is no finite simple group $S$ with $\text{m}_p(S)=2p^2-p-1$.
\end{maintheorem}

Theorem~\ref{simple} is the heart of the paper.

\medskip
\noindent\textbf{Outline.}
Section~\ref{sec:prelim} fixes the notation and proves two general lemmas valid for groups satisfying $\text{s}_p(G)=2p+1$. Section~\ref{sec:simple} shows that a nonabelian simple group $S$ with $\text{s}_p(S)=2p+1$ acts faithfully and $2$-transitively on $2p+1$ points, with $|S|_p=p$ and $|S|$ small. Section~\ref{sec:thmB} uses the classification of $2$-transitive groups to complete the proof of Theorem~\ref{simple}. Section~\ref{sec:general} uses Theorem~\ref{simple} to control arbitrary groups with $\text{s}_p(G)=2p+1$: the Sylow $p$-subgroups have order $p$, $\mathrm{O}_{p'}(G)$ together with a Sylow $p$-subgroup is normal, and $2p+1$ is a prime power. Section~\ref{sec:proof} contains the arithmetic of prime powers $2p+1$, the construction of examples, and the proof of Theorem~\ref{main}.

\section{Notation and preliminaries}\label{sec:prelim}

All groups are finite, and $p$ denotes a prime. For a group $G$ let $\Omega(G)$ be the set of subgroups of order $p$ of $G$, so $\text{s}_p(G)=|\Omega(G)|$; we write $\Omega=\Omega(G)$ when $G$ is clear. The group $G$ acts on $\Omega$ by conjugation, $g\cdot Q=gQg^{-1}$, and the stabilizer of $Q$ is $\mathrm{N}_G(Q)$. Write $|G|_p$ for the largest power of $p$ dividing $|G|$, $\mathrm{Syl}_p(G)$ for the set of Sylow $p$-subgroups of $G$, and $\mathrm{O}_{p'}(G)$ for the largest normal $p'$-subgroup of $G$, where $p'$ means that the order of the subgroup is not divisible by $p$. 

\medskip
In this section, we record a remark on fixed points and two lemmas: the first holds for $\text{s}_p(G)=kp+1$ with $k\in\{1,2\}$, the second for $\text{s}_p(G)=kp+1$ with $1\le k<p$ ($p$ odd); both are applied below with $\text{s}_p(G)=2p+1$.

\begin{rem}\label{rem:fixed}
Let $P\in\mathrm{Syl}_p(G)$ and $Q\in\Omega$. If $P$ normalizes $Q$, then $PQ$ is a $p$-group containing $P$, so $PQ=P$ and $Q\le P$. Hence,  the fixed points of $P$ on $\Omega$ are exactly the normal subgroups of order $p$ of $P$; if $P$ is abelian, these are all the subgroups of order $p$ of $P$.
\end{rem}

The first lemma describes how $\text{s}_p$ behaves under passage to a quotient by a normal $p'$-subgroup.

\begin{lem}\label{lem:Q}
Let $\text{s}_p(G)=kp+1$ with $k\in\{1,2\}$, and let $N\trianglelefteq G$ with $p\nmid|N|$. Then $\text{s}_p(G/N)\in\{1,kp+1\}$.
\end{lem}

\begin{proof}
Let $\Omega=\Omega(G)$, so $|\Omega|=kp+1$. For $Q\in\Omega$ we have $Q\cap N=1$, so $QN/N$ has order $p$. Conversely, for each subgroup $J\le G/N$ of order $p$ let $G_J$ be its preimage. Then $|G_J|=p|N|$, so the Sylow $p$-subgroups of $G_J$ have order $p$. Put $n_J=|\Omega\cap G_J|$. By Sylow's theorem $n_J\equiv1\pmod p$. The sets $\Omega\cap G_J$ partition $\Omega$ (each $Q$ determines $J=QN/N$), so
\[
\sum_J n_J=kp+1.
\]
If $n_J>1$ then $n_J\ge p+1$; since $2(p+1)>kp+1$ for $k\in\{1,2\}$, at most one $J$ has $n_J>1$. Put $\mathcal A=\{J:n_J>1\}$, $c=|\{J:n_J=1\}|$, $t=\text{s}_p(G/N)$. Then $t=c+|\mathcal A|$, $|\mathcal A|\le1$, $t\ge1$, and $t\equiv1\pmod p$ by Frobenius.

\medskip
We claim that $J\mapsto Q$, where $Q$ is the unique element of $\Omega\cap G_J$, is a bijection from $\{J:n_J=1\}$ onto $\Omega(\mathrm{C}_G(N))$. If $n_J=1$, then $Q$ is characteristic in $G_J=NQ$ (conjugation by $G_J$ permutes its subgroups of order $p$), so $[n,q]=n^{-1}(q^{-1}nq)\in N$ and $[n,q]=(n^{-1}q^{-1}n)q\in Q$ for all $n\in N$, $q\in Q$; hence $[N,Q]\le N\cap Q=1$, that is, $Q\le \mathrm{C}_G(N)$. Conversely, if $Q\in\Omega$ and $Q\le\mathrm{C}_G(N)$, then $Q\trianglelefteq NQ$, so $Q$ is the unique Sylow $p$-subgroup of $NQ$ and $n_{QN/N}=1$. Injectivity holds as $J=QN/N$ is determined by $Q$.

\medskip
If $c\ge1$, then $p\mid|\mathrm{C}_G(N)|$ and Frobenius gives $c\equiv1\pmod p$. Hence $|\mathcal A|=t-c\equiv0\pmod p$, so $\mathcal A=\varnothing$ and $t=c=\sum_Jn_J=kp+1$. If $c=0$, then $t=|\mathcal A|\le1$, so $t=1$.
\end{proof}

The second lemma restricts the Sylow $p$-subgroups and excludes normal subgroups of order $p$.

\begin{lem}\label{lem:sylow}
Let $p$ be an odd prime and $G$ a group with $\text{s}_p(G)=kp+1$ for an integer
$1\le k<p$, and let $P\in\mathrm{Syl}_p(G)$.
\begin{enumerate}[label=(\roman*)]
\item If $|P|\ge p^2$, then $P$ is noncyclic and $\text{s}_p(P)=p+1$.
\item If $k\ge 2$, then $G$ has no normal subgroup of order $p$.
\end{enumerate}
\end{lem}

\begin{proof}
(i) Suppose $P$ is cyclic, with $Q_0$ its subgroup of order $p$. By Remark~\ref{rem:fixed}, $Q_0$ is the only point of $\Omega$ fixed by $P$. Since $|\Omega\setminus\{Q_0\}|=kp<p^2$, all other $P$-orbits have size $p$, with stabilizer the unique subgroup $P_1$ of index $p$ in $P$. As $|P|\ge p^2$, $Q_0\le P_1$, and $P_1$ is the stabilizer of every point of each such orbit ($P$ is abelian), so $P_1$ fixes $\Omega$ pointwise. In particular $Q_0$ normalizes every $Q\in\Omega$. For $Q\ne Q_0$, $Q_0Q$ is then a $p$-group with two subgroups of order $p$, hence noncyclic; but it lies in a conjugate of the cyclic group $P$, a contradiction. So $P$ is noncyclic, and by Kulakoff's theorem \cite{Kulakoff}, $\text{s}_p(P)\equiv p+1\pmod{p^2}$. As $\text{s}_p(P)\le kp+1<p^2+p+1$, we get $\text{s}_p(P)=p+1$.

\medskip
(ii) Suppose $Z\trianglelefteq G$ has order $p$. If $|G|_p=p$, then $Z$ is the unique
Sylow $p$-subgroup and $\text{s}_p(G)=1$, a contradiction; so $|P|\ge p^2$ and (i)
applies. For $Q\in\Omega\setminus\{Z\}$, $ZQ\cong\mathbb Z_p\times\mathbb Z_p$ contains
exactly $p$ members of $\Omega\setminus\{Z\}$. Since $|\Omega\setminus\{Z\}|=kp$, the
relation $Q\sim Q'\iff ZQ=ZQ'$ has exactly $k$ classes $\mathcal C_1,\dots,\mathcal C_k$,
with $E_i=ZQ_i$ for $Q_i\in\mathcal C_i$. The group $G$ permutes the $k$ classes; let $X$
be the kernel of this action. Then $|G:X|$ divides $k!$, which is coprime to $p$ because
$k<p$, so $P\le X$, and $P$ is a Sylow $p$-subgroup of $X$. Each $E_i$ is normalized by
$X$, because $(ZQ_i)^x=ZQ_i^x$ and $Q_i^x\in\mathcal C_i$. So $E_i$ is a normal
$p$-subgroup of $X$, hence $E_i\le P$. Then
$\Omega\subseteq\{Z\}\cup E_1\cup\dots\cup E_k\subseteq P$, so
$\text{s}_p(P)=kp+1\ne p+1$ since $k\ge2$, contradicting (i).
\end{proof}

\section{Simple groups}\label{sec:simple}

Throughout this section $S$ is a nonabelian finite simple group with $\text{s}_p(S)=2p+1$, and $\Omega=\Omega(S)$. (Abelian simple groups $\mathbb Z_r$ have $\text{s}_p\le1$.) We show that $S$ acts faithfully on $\Omega$, that $|S|_p=p$, that $|S|\le\frac12n(n-1)(n-3)$ with $n=2p+1$, and that $S$ is $2$-transitive on $\Omega$. These facts are used in Section~\ref{sec:thmB} to reach a contradiction.

\medskip
We begin with faithfulness.

\begin{lem}\label{faith}
The conjugation action of $S$ on $\Omega$ is faithful.
\end{lem}

\begin{proof}
The kernel of the action is normal in $S$, so it is $1$ or $S$. If it were $S$, every subgroup of order $p$ would be normal in $S$; being a proper nontrivial subgroup of the nonabelian simple group $S$, this is impossible.
\end{proof}
Note that faithfulness implies that $S$ embeds in $\mathrm{Sym}(\Omega)\cong\mathrm{Sym}(2p+1)$ and $|S|$ divides $(2p+1)!$. If $p=2$, then $S\le\mathrm{Sym}(5)$ is nonabelian simple, so $S\cong A_5$, which has $15\ne5$ involutions. Hence, from now on $p\ge3$.

\medskip
Next we show that the Sylow $p$-subgroups of $S$ have order $p$.

\begin{prop}\label{prop:Sp}
Let $p\ge3$. Then $|S|_p=p$. Consequently:
\begin{enumerate}[label=(\alph*)]
\item every subgroup of order $p$ is a Sylow $p$-subgroup;
\item $S$ acts transitively on $\Omega$;
\item for $Q\in\Omega$ the stabilizer of $Q$ is $\mathrm{N}_S(Q)$, of index $2p+1$.
\end{enumerate}
\end{prop}

\begin{proof}
Let $P\in\mathrm{Syl}_p(S)$. By Lemma~\ref{faith}, $|P|\le|\mathrm{Sym}(2p+1)|_p=p^2$. Suppose $|P|=p^2$. According to Lemma~\ref{lem:sylow}(i), $P$ is not cyclic, so $P\cong\mathbb Z_p\times\mathbb Z_p$. By Remark~\ref{rem:fixed}, $P$ fixes exactly its $p+1$ subgroups of order $p$, and the remaining $p$ points of $\Omega$ form a single $P$-orbit. The image of $P$ in $\mathrm{Sym}(p)$ has order at most $p$, so some $1\ne x\in P$ acts trivially on this orbit; it also fixes the other $p+1$ points, so $x$ acts trivially on $\Omega$, contradicting Lemma~\ref{faith}. Hence, $|P|=p$.

\medskip
Then (a) is clear; (b) follows since the Sylow subgroups are conjugate; (c) is the orbit-stabilizer theorem.
\end{proof}

Now fix $P=\langle x\rangle\in\Omega$ and consider it as a base point. Put $H=\mathrm{N}_S(P)$, $C=\mathrm{C}_S(P)$, and  $\Omega'=\Omega\setminus\{P\}$. Hence,  $H$ is the stabilizer of $P$, and $|S|=(2p+1)|H|$. We first describe the orbits of $P$ on $\Omega'$.

\begin{lem}\label{orb}
$P$ has exactly two orbits $O_1,O_2$ on $\Omega'$, each of size $p$, and every $h\in H$ permutes $\{O_1,O_2\}$.
\end{lem}

\begin{proof}
By Remark~\ref{rem:fixed} and $|P|=p$, $P$ fixes only the point $P$, so all $P$-orbits on $\Omega'$ have size $p$; as $|\Omega'|=2p$ there are exactly two. For $h\in H$ we have $hxh^{-1}=x^{j}$ for some $j$, hence $h\cdot(x^k\cdot Q)=x^{jk}\cdot(h\cdot Q)$, so $h\cdot\mathrm{Orb}_P(Q)=\mathrm{Orb}_P(h\cdot Q)$, where $\mathrm{Orb}_P(Q)$ denotes the $P$-orbit of $Q\in\Omega'$.
\end{proof}

To bound $|H|$ we need an elementary fact about permutations.

\begin{lem}\label{fullcyc}
Let $\rho,\sigma$ be permutations of a finite set $\Delta$, with $\rho$ a full cycle. If $\rho\sigma=\sigma\rho$ then $\sigma=\rho^\lambda$ for some $\lambda\in\mathbb Z$.
\end{lem}

\begin{proof}
Write $\Delta=\{a_0,\dots,a_{m-1}\}$ with $\rho(a_i)=a_{(i+1)\bmod m}$. Choose $\lambda$ with $\sigma(a_0)=\rho^\lambda(a_0)$. Then $\sigma(a_i)=\sigma(\rho^i(a_0))=\rho^i(\sigma(a_0))=\rho^\lambda(a_i)$ for all $i$.
\end{proof}

This yields the following bound on the order of $S$.

\begin{lem}\label{bound}
$|H|\le2p(p-1)$. Consequently, with $n=2p+1$,
\begin{equation}\label{eq:bound}
|S|\le 2p(p-1)(2p+1)=n(n-1)(n-3)/2.
\end{equation}
\end{lem}

\begin{proof}
The group $H/C$ embeds in $\mathrm{Aut}(P)\cong\mathrm{Aut}(\mathbb Z_p)$, so $|H:C|\le p-1$. By Lemma~\ref{orb}, $C\trianglelefteq H$ permutes $\{O_1,O_2\}$; let $D$ be the kernel of this action, so $|C:D|\le2$ and $|H|\le2(p-1)|D|$. It remains to show $D=P$.

\medskip
Clearly $P\le D$. Let $d\in D$. Then $d$ fixes the point $P$, commutes with $x$, and preserves each $O_i$, on which $x$ acts as a full $p$-cycle. By Lemma~\ref{fullcyc}, $d$ acts on $O_i$ as $x^{\lambda_i}$ for some $\lambda_i\in\mathbb Z_p$ ($i=1,2$). Hence $d\mapsto(\lambda_1,\lambda_2)$ is a homomorphism $D\to\mathbb Z_p^2$, injective because the action of $S$ on $\Omega$ is faithful. So $D$ is a $p$-group; as $|S|_p=p$, we get $D=P$. Finally, $|S|=(2p+1)|H|$ and $p=(n-1)/2$, $2p-2=n-3$ give the stated form.
\end{proof}

Finally we show that $S$ is $2$-transitive on $\Omega$.

\begin{lem}\label{2trans}
For $p\ge3$, $S$ is $2$-transitive on $\Omega$.
\end{lem}

\begin{proof}
By transitivity of $S$ on $\Omega$, it suffices to show that $H$ is transitive on $\Omega'$ by \cite{Cam99}. Each $H$-orbit on $\Omega'$ is a union of $P$-orbits, and $H$ permutes $\{O_1,O_2\}$ (Lemma~\ref{orb}). So if $H$ is not transitive on $\Omega'$, its orbits are exactly $O_1$ and $O_2$. Assume this.

\medskip
For $u\ne v$ in $\Omega$ choose $g\in S$ with $g\cdot P=u$, and say that $v$ has \emph{type $i$ relative to $u$} if $g^{-1}\cdot v\in O_i$. This is independent of $g$: a different choice is $gh$ with $h\in H$, and $h^{-1}g^{-1}\cdot v\in O_i$ as $H$ preserves $O_i$. The notion is $S$-invariant: $v$ has type $i$ relative to $u$ iff $s\cdot v$ has type $i$ relative to $s\cdot u$. Each point has exactly $p$ points of each type.

\medskip
Fix $i$ and call $v\in O_i$ \emph{mutual} if $P$ has type $i$ relative to $v$. Since $H$ is transitive on $O_i$ and fixes $P$, by invariance either all $v\in O_i$ are mutual or none is. Suppose all are. By $S$-transitivity and invariance, the same holds at every point, so ``type $i$'' is a symmetric relation, that is, an undirected graph on $2p+1$ vertices, regular of degree $p$. The degree sum $(2p+1)p$ is odd (both factors are odd), yet it equals twice the number of edges, a contradiction. Hence for no $i$ and no $u\ne v$ is $v$ of type $i$ relative to $u$ while $u$ is of type $i$ relative to $v$.

\medskip
Write $u\to v$ if $v$ has type $1$ relative to $u$. For $u\ne v$, by the previous paragraph the types of $v$ relative to $u$ and of $u$ relative to $v$ are different, so exactly one of $u\to v$, $v\to u$ holds. Hence,  $\to$ is a tournament on $\Omega$, and by invariance $S$ acts on it by automorphisms. (Compare \cite{Moon68}: tournaments have automorphism groups of odd order.)

\medskip
By the Feit--Thompson theorem \cite{FT63}, $|S|$ is even, so $S$ contains an involution $t$. If $t$ moved some $u$ to $v=t\cdot u\ne u$, then $t\cdot v=u$, and $u\to v$ would imply $v\to u$ (and vice versa), contradicting the tournament property. So $t$ fixes every point of $\Omega$, contradicting Lemma~\ref{faith}.
\end{proof}

\section{Proof of Theorem~\ref{simple}}\label{sec:thmB}

Suppose $S$ is as in Section~\ref{sec:simple}. By the results there, $S$ is a nonabelian simple group with a faithful $2$-transitive action on $n=2p+1$ points, $p\ge3$ prime, $|S|_p=p$, and $|S|\le\frac12n(n-1)(n-3)$ by \eqref{eq:bound}. (The case $p=2$ was settled in Section~\ref{sec:simple}.) We now use the classification of the finite $2$-transitive groups, in the form (see \cite[\S7.7]{DM96}): if a nonabelian simple group $S$ has a faithful $2$-transitive action on $n$ points, then $S$ and $n$ appear in the following list.
\begin{enumerate}[label=(\arabic*), itemsep=1pt]
\item $A_n$ in its natural action ($n\ge5$);
\item $\mathrm{PSL}_d(q)$, $d\ge2$, $(d,q)\ne(2,2),(2,3)$, on the points of projective space, $n=\frac{q^d-1}{q-1}$;
\item $\mathrm{PSU}_3(q)$, $q\ge3$, $n=q^3+1$;
\item $\mathrm{Sz}(q)$, $q=2^{2m+1}\ge8$, $n=q^2+1$;
\item the Ree group ${}^2G_2(q)$, $q=3^{2m+1}\ge27$, $n=q^3+1$;
\item $\mathrm{Sp}_{2m}(2)$, $m\ge3$, $n=2^{m-1}(2^m\pm1)$;
\item sporadic cases: $\mathrm{PSL}_2(11)$ ($n=11$), $A_7$ ($n=15$), $M_{11}$ ($n=11,12$), $M_{12}$ ($n=12$), $M_{22}$ ($n=22$), $M_{23}$ ($n=23$), $M_{24}$ ($n=24$), $\mathrm{HS}$ ($n=176$), $\mathrm{Co}_3$ ($n=276$).
\end{enumerate}
Here $n=2p+1$ is odd, so all cases of even degree are excluded; this removes (5), (6) and the even-degree sporadic cases. In (3) and (4) with $n$ odd, $q$ is even, and $n=2p+1$ gives $q^3=2p$, respectively $q^2=2p$; as $4\mid q^2$ and $8\mid q^3$, this would give $2\mid p$, impossible.

\medskip
\noindent\emph{Case (7).} The remaining odd degrees are $11$ ($\mathrm{PSL}_2(11)$, $M_{11}$), $15$ ($A_7$) and $23$ ($M_{23}$). The bound $\frac12n(n-1)(n-3)$ equals $440$, $1260$, $5060$ for $n=11,15,23$ respectively, while $|\mathrm{PSL}_2(11)|=660$, $|M_{11}|=7920$, $|A_7|=2520$ and $|M_{23}|=10200960$. In each case $|S|$ exceeds the bound, contradicting \eqref{eq:bound}.

\medskip
\noindent\emph{Case (1).} $A_{2p+1}$ contains two disjoint $p$-cycles (even permutations, as $p$ is odd), which generate a subgroup $\mathbb Z_p\times\mathbb Z_p$, contradicting $|S|_p=p$.

\medskip
\noindent\emph{Case (2).} Here
\[
2p+1=\frac{q^d-1}{q-1}=1+q+\cdots+q^{d-1},\quad\text{so}\quad 2p=q\,e,\ \ e=1+q+\cdots+q^{d-2}.
\]
If $q>2$ and $d=2$, then $q=2p$ is a prime power, impossible as $p$ is odd. If $q>2$ and $d\ge3$, then $q\ge3$ and $e\ge q+1\ge4$, but the only divisors of $2p$ are $1,2,p,2p$, so $q\in\{p,2p\}$ forces $e\le2$, a contradiction. If $q=2$, then $d\ge3$ and $n=2^d-1$. For $d=3$, $|S|=168>84=\frac12\cdot7\cdot6\cdot4$. For $d\ge4$, $|S|=|\mathrm{GL}_d(2)|>2^{d^2-2}\ge2^{3d-1}>\frac12n^3$ (using $\prod_{i\ge1}(1-2^{-i})>\frac14$ and $d^2-3d-1\ge0$), again contradicting \eqref{eq:bound}.

\medskip
All cases are impossible, so no such $S$ exists. This proves Theorem~\ref{simple}. \qed

\section{Arbitrary finite groups}\label{sec:general}

We now use Theorem~\ref{simple} to control arbitrary groups with $\text{s}_p(G)=2p+1$. We show that the Sylow $p$-subgroups have order $p$, that a Sylow $p$-subgroup together with $\mathrm{O}_{p'}(G)$ generates a normal subgroup, and that $2p+1$ is a prime power. First, the Sylow $p$-subgroups have order $p$.

\begin{prop}\label{Gp}
Let $p$ be an odd prime and $G$ a group with $\text{s}_p(G)=2p+1$. Then $|G|_p=p$.
\end{prop}

\begin{proof}
Let $G$ be a counterexample of minimal order, so $|P|\ge p^2$ for $P\in\mathrm{Syl}_p(G)$; by Lemma~\ref{lem:sylow}, $P$ is noncyclic, $\text{s}_p(P)=p+1$, and $G$ has no normal subgroup of order $p$. By Proposition~\ref{prop:Sp}, $G$ is not simple (an abelian simple group has $\text{s}_p\le1$). So let $1\ne N\lneq G$ be normal.

\medskip
\emph{Case $p\nmid|N|$.} By Lemma~\ref{lem:Q}, $\text{s}_p(G/N)\in\{1,2p+1\}$. Since $P\cap N=1$, $P$ embeds in $G/N$, so $\text{s}_p(G/N)\ge\text{s}_p(P)=p+1>1$. Hence,  $\text{s}_p(G/N)=2p+1$ and $|G/N|_p=|G|_p\ge p^2$, so $G/N$ is a smaller counterexample, a contradiction.

\medskip
\emph{Case $p\mid|N|$.} By Frobenius, $\text{s}_p(N)\in\{1,p+1,2p+1\}$.
\begin{enumerate}
\item If $\text{s}_p(N)=1$, the subgroup of order $p$ of $N$ is characteristic in $N$, hence normal in $G$, a contradiction.
\item If $\text{s}_p(N)=2p+1$, then $\Omega\subseteq N$. If $|N|_p\ge p^2$, $N$ is a smaller counterexample. If $|N|_p=p$, then $|P\cap N|=p$ and $P\cap N$ contains every subgroup of order $p$ of $P$, so $\text{s}_p(P)=1$, a contradiction.
\item If $\text{s}_p(N)=p+1$, the set $\Omega_{\rm out}=\{Q\in\Omega:Q\not\le N\}$ has size $p$ and is $G$-invariant. Take $Q\in\Omega_{\rm out}$. It acts on the $p+1\equiv1\pmod p$ subgroups of order $p$ of $N$, so it normalizes one of them, $Q_1$. Then $E=QQ_1\cong\mathbb Z_p\times\mathbb Z_p$ and $E\cap N=Q_1$. The $p$ subgroups of order $p$ of $E$ other than $Q_1$ lie outside $N$, so they are exactly $\Omega_{\rm out}$; they generate $E$, so $E=\langle\Omega_{\rm out}\rangle\trianglelefteq G$. Hence $Q_1=E\cap N\trianglelefteq G$, again a contradiction.\qedhere
\end{enumerate}
\end{proof}

Next we show that a Sylow $p$-subgroup together with $\mathrm{O}_{p'}(G)$ generates a normal subgroup.

\begin{lem}\label{lem:E1}
Let $p$ be an odd prime, $G$ a group with $\text{s}_p(G)=2p+1$, $P\in\mathrm{Syl}_p(G)$, $K=\mathrm{O}_{p'}(G)$ and $\bar G=G/K$. Then $\mathrm{O}_{p'}(\bar G)=1$, $|\bar G|_p=p$, and $PK\trianglelefteq G$.
\end{lem}

\begin{proof}
By Proposition~\ref{Gp}, $|G|_p=p$, so $|\bar G|_p=p$ and $\bar G\ne1$. The preimage of a normal $p'$-subgroup of $\bar G$ is a normal $p'$-subgroup of $G$, hence lies in $K$; so $\mathrm{O}_{p'}(\bar G)=1$. Let $\bar P=PK/K\in\mathrm{Syl}_p(\bar G)$; we show $\bar P\trianglelefteq\bar G$.

\medskip
Let $L$ be a minimal normal subgroup of $\bar G$. If $L$ is abelian, it is an elementary abelian $r$-group with $r=p$ (as $\mathrm{O}_{p'}(\bar G)=1$), so $|L|\le p$ and $L=\bar P\trianglelefteq\bar G$.

\medskip
Suppose $L$ is nonabelian, so $L=T^j$ with $T$ a nonabelian simple group. Then $p\mid|L|$ (otherwise $L\le\mathrm{O}_{p'}(\bar G)=1$), so $p\mid|T|$. Since $|T|_p^{\,j}=|L|_p\le|\bar G|_p=p$, we get $j=1$, $L=T$, and $|T|_p=p$. By Lemma~\ref{lem:Q} applied to the normal $p'$-subgroup $K$ of $G$, $\text{s}_p(\bar G)\in\{1,2p+1\}$. If $\text{s}_p(\bar G)=1$, $\bar P$ is the unique subgroup of order $p$, hence normal, and we are done. Assume $\text{s}_p(\bar G)=2p+1$. By Frobenius, $\text{s}_p(T)\equiv1\pmod p$ and $\text{s}_p(T)\le2p+1$, so $\text{s}_p(T)\in\{1,p+1,2p+1\}$.
\begin{enumerate}
 \item $\text{s}_p(T)=1$ is impossible: a unique subgroup of order $p$ would be a normal Sylow subgroup of the nonabelian simple group $T$.
\item $\text{s}_p(T)=2p+1$ is impossible by Theorem~\ref{simple}.
\item If $\text{s}_p(T)=p+1$, pick $Q\in\Omega(\bar G)\setminus\Omega(T)$ (possible as $2p+1>p+1$). Then $Q$ acts on $\Omega(T)$, and since $p+1\equiv1\pmod p$ it normalizes some $Q_1\in\Omega(T)$. As $Q\not\le T$, $Q\cap Q_1=1$, so $|QQ_1|=p^2$, contradicting $|\bar G|_p=p$.
\end{enumerate}

So $\text{s}_p(\bar G)=2p+1$ cannot occur, and $\bar P\trianglelefteq\bar G$, that is, $PK\trianglelefteq G$.
\end{proof}

Finally, we show that $2p+1$ is a prime power.
\begin{prop}\label{prop:primepower}
Let $p$ be an odd prime and $G$ a group with $\text{s}_p(G)=2p+1=q$. Then $q=r^a$ for a prime $r$ and some $a\ge1$.
\end{prop}

\begin{proof}
Let $P\in\mathrm{Syl}_p(G)$ and $K=\mathrm{O}_{p'}(G)$. By Proposition~\ref{Gp}, $|G|_p=p$, so every subgroup of order $p$ is Sylow.

\medskip
First, we show that $q=|K:\mathrm{C}_K(P)|$. By Lemma~\ref{lem:E1}, $PK\trianglelefteq G$, so every $Q\in\Omega$ (a conjugate of $P$) lies in $PK$. Hence $\Omega$ is exactly the set of Sylow $p$-subgroups of $PK$. Since $P\cap K=1$, $\mathrm{N}_K(P)=\mathrm{C}_K(P)$: for $g\in\mathrm{N}_K(P)$ and $y\in P$, $[g,y]=g^{-1}(y^{-1}gy)\in K$ as $K\trianglelefteq G$, and $[g,y]=(g^{-1}y^{-1}g)y\in P$ as $g$ normalizes $P$, so $[g,y]\in K\cap P=1$. By Dedekind's modular law \cite[Problem~2.9]{Isaacs94},
\[
\mathrm{N}_{PK}(P)=P\,\mathrm{N}_K(P)=P\,\mathrm{C}_K(P),
\]
of order $p\,|\mathrm{C}_K(P)|$. By Sylow's theorem,
\begin{equation}\label{eq:index}
q=\text{s}_p(G)=\frac{|PK|}{|P\mathrm{C}_K(P)|}=|K:\mathrm{C}_K(P)|.
\end{equation}
In particular $K\ne1$ since $q>1$.

\medskip
We now show that $q$ is a prime power. Since $P$ is a $p$-group acting on the $p'$-group $K$, the action is coprime. By the standard facts on coprime action (see \cite[Theorem~3.23a and Lemma~3.32]{Isaacs08}), for each prime $\ell\mid|K|$ there is a $P$-invariant Sylow $\ell$-subgroup $K_\ell$ of $K$, and $\mathrm{C}_{K_\ell}(P)=K_\ell\cap\mathrm{C}_K(P)$ is a Sylow $\ell$-subgroup of $\mathrm{C}_K(P)$. So the $\ell$-part of \eqref{eq:index} is
\[
|K_\ell:\mathrm{C}_{K_\ell}(P)|=:\ell^{a_\ell},\quad q=\prod_\ell\ell^{a_\ell}.
\]
If $a_\ell>0$, put $C_\ell=\mathrm{C}_{K_\ell}(P)$. The group $P$ acts on $K_\ell$ with fixed-point set $C_\ell$ and all other orbits of size $p$, so $p\mid|K_\ell|-|C_\ell|=|C_\ell|(\ell^{a_\ell}-1)$. As $|C_\ell|$ is a power of $\ell\ne p$, we get $p\mid\ell^{a_\ell}-1$, so $\ell^{a_\ell}\ge p+1$.

\medskip
If two distinct primes had $a_\ell>0$, then $q\ge(p+1)^2>2p+1=q$, a contradiction. Since $q>1$, exactly one prime $r$ has $a_r>0$ (and $r\ne p$), so $q=r^a$ with $a=a_r\ge1$.
\end{proof}

\section{Proof of Theorem~\ref{main}}\label{sec:proof}

In this section we determine which primes $p$ admit a prime power $2p+1$, construct examples, and combine this with the results of Sections~\ref{sec:thmB} and~\ref{sec:general}.

\medskip
We start with the arithmetic of prime powers of the form $2p+1$.

\begin{prop}\label{arith}
Let $r$ be a prime and $f\ge1$ an integer. If $p=(r^f-1)/2$ is prime, then either
\begin{enumerate}[label=(\arabic*)]
\item $f=1$ and $r=2p+1$ is an odd prime, or
\item $r=3$ and $f>1$ is an odd prime.
\end{enumerate}
\end{prop}

\begin{proof}
Since $2p=r^f-1$ is even, $r$ is odd. If $f$ were even, then $2p=(r^{f/2}-1)(r^{f/2}+1)$. However, \(r^{f/2}-1\) and \(r^{f/2}+1\) are consecutive even integers, one of which is divisible by \(4\) so $4\mid p$, a contradiction. Hence $f$ is odd and
\[
p=\frac{r-1}{2}\bigl(1+r+\cdots+r^{f-1}\bigr).
\]
The second factor is odd and at least $1$, so since $p$ is prime, one of the factors is $1$: either $f=1$ (case (1)), or $(r-1)/2=1$, that is, $r=3$ and $p=(3^f-1)/2$ with $f>1$. In the latter case, if $f=ab$ with $a,b>1$, then
\[
\frac{3^f-1}{2}=\frac{3^a-1}{2}\bigl(1+3^a+\cdots+3^{(b-1)a}\bigr)
\]
is a product of two integers strictly greater than $1$, contradicting primality. So $f$ is an odd prime.
\end{proof}

We can now prove Theorem~\ref{main}.

\begin{proof}[Proof of Theorem~\ref{main}]
Recall that $\text{m}_p(G)=2p^2-p-1$ if and only if $\text{s}_p(G)=2p+1$.

\medskip
(b)$\Leftrightarrow$(c): If $2p+1=r^f$ is a prime power, Proposition~\ref{arith} gives (c). Conversely, an odd prime or $3^f$ is a prime power.

\medskip
(a)$\Rightarrow$(b): Let $\text{s}_p(G)=2p+1$. If $p=2$, then $2p+1=5$ is prime. If $p$ is odd, Proposition~\ref{prop:primepower} shows that $2p+1$ is a prime power.

\medskip
(b)$\Rightarrow$(a): Let $q=2p+1=r^f$. Then $p\mid q-1$, so $\mathbb F_q^\times$ contains an element $\zeta$ of order $p$. Let $P=\langle\zeta\rangle\cong\mathbb Z_p$ act on $V=(\mathbb F_q,+)\cong\mathbb Z_r^f$ by multiplication, and let $G=V\rtimes P$. The action is fixed-point-free, so $\mathrm{N}_G(P)=P$, and every subgroup of order $p$ is a Sylow $p$-subgroup (as $|G|_p=p$) conjugate to $P$. Hence $\text{s}_p(G)=|G:P|=q=2p+1$ and $\text{m}_p(G)=(p-1)(2p+1)=2p^2-p-1$. (For $p=2$ this is $D_{10}$.) More generally $(V\rtimes P)\times B$ works for any group $B$ with $p\nmid|B|$.

\end{proof}

\textbf{AI statement.}
The free version of Claude assisted with literature review, proof auditing, and error correction. All mathematical arguments and their verification remain the authors’ responsibility.

\bibliographystyle{alpha}
\bibliography{Ref}

@book {Moon68,
    AUTHOR = {Moon, John W.},
     TITLE = {Topics on tournaments},
 PUBLISHER = {Holt, Rinehart and Winston, New York-Montreal, Que.-London},
      YEAR = {1968},
     PAGES = {viii+104},
   MRCLASS = {05.40},
  MRNUMBER = {256919},
MRREVIEWER = {W.\ Moser},
}

@article {Asboei26,
    AUTHOR = {Asboei, A. K.},
     TITLE = {Finite groups with few elements of prime order},
   JOURNAL = {New York J. Math.},
  FJOURNAL = {New York Journal of Mathematics},
    VOLUME = {32},
      YEAR = {2026},
     PAGES = {1344--1349},
}

@article {Frobenius95,
    AUTHOR = {Frobenius, Ferdinand G.},
     TITLE = {Verallgemeinerung des {Sylow}schen Satzes},
   JOURNAL = {Sitzungsber. K{\"o}niglich Preuss. Akad. Wiss. Berlin},
      YEAR = {1895},
     PAGES = {981--993},
      DOI = {10.3931/e-rara-18880},
}

@article {FT63,
    AUTHOR = {Feit, Walter and Thompson, John G.},
     TITLE = {Solvability of groups of odd order},
   JOURNAL = {Pacific J. Math.},
  FJOURNAL = {Pacific Journal of Mathematics},
    VOLUME = {13},
      YEAR = {1963},
     PAGES = {775--1029},
      ISSN = {0030-8730,1945-5844},
   MRCLASS = {20.40 (20.25)},
  MRNUMBER = {166261},
MRREVIEWER = {M.\ Suzuki},
       URL = {http://projecteuclid.org/euclid.pjm/1103053943},
}

@book {DM96,
    AUTHOR = {Dixon, John D. and Mortimer, Brian},
     TITLE = {Permutation groups},
    SERIES = {Graduate Texts in Mathematics},
    VOLUME = {163},
 PUBLISHER = {Springer-Verlag, New York},
      YEAR = {1996},
     PAGES = {xii+346},
      ISBN = {0-387-94599-7},
   MRCLASS = {20B05 (20-01 20B07)},
  MRNUMBER = {1409812},
MRREVIEWER = {Martin\ W.\ Liebeck},
       DOI = {10.1007/978-1-4612-0731-3},
       URL = {https://doi.org/10.1007/978-1-4612-0731-3},
}

@book {Isaacs08,
    AUTHOR = {Isaacs, I. Martin},
     TITLE = {Finite group theory},
    SERIES = {Graduate Studies in Mathematics},
    VOLUME = {92},
 PUBLISHER = {American Mathematical Society, Providence, RI},
      YEAR = {2008},
     PAGES = {xii+350},
      ISBN = {978-0-8218-4344-4},
   MRCLASS = {20Dxx (20-01)},
  MRNUMBER = {2426855},
MRREVIEWER = {Ronald\ Solomon},
       DOI = {10.1090/gsm/092},
       URL = {https://doi.org/10.1090/gsm/092},
}

@article {Kulakoff,
    AUTHOR = {Kulakoff, A.},
     TITLE = {\"Uber die {A}nzahl der eigentlichen {U}ntergruppen und der
              {E}lemente von gegebener {O}rdnung in {$p$}-{G}ruppen},
   JOURNAL = {Math. Ann.},
  FJOURNAL = {Mathematische Annalen},
    VOLUME = {104},
      YEAR = {1931},
    NUMBER = {1},
     PAGES = {778--793},
      ISSN = {0025-5831,1432-1807},
   MRCLASS = {99-04},
  MRNUMBER = {1512698},
       DOI = {10.1007/BF01457969},
       URL = {https://doi.org/10.1007/BF01457969},
}

@book {Cam99,
    AUTHOR = {Cameron, Peter J.},
     TITLE = {Permutation groups},
    SERIES = {London Mathematical Society Student Texts},
    VOLUME = {45},
 PUBLISHER = {Cambridge University Press, Cambridge},
      YEAR = {1999},
     PAGES = {x+220},
      ISBN = {0-521-65302-9; 0-521-65378-9},
   MRCLASS = {20B40},
  MRNUMBER = {1721031},
MRREVIEWER = {Miguel\ \'A.\ Borges-Trenard},
       DOI = {10.1017/CBO9780511623677},
       URL = {https://doi.org/10.1017/CBO9780511623677},
}

@book {Isaacs94,
    AUTHOR = {Isaacs, I. Martin},
     TITLE = {Algebra: a graduate course},
 PUBLISHER = {Brooks/Cole Publishing Company, Pacific Grove, CA},
      YEAR = {1994},
     PAGES = {xii+516},
      ISBN = {0-534-19002-2},
   MRNUMBER = {1276273},
}

\end{document}